\documentclass[11pt]{article}

\usepackage{amsmath, amsfonts, amsthm, amsopn, mathrsfs, xcolor}
\usepackage{hyperref, cleveref}
\usepackage{graphicx, epstopdf}
\usepackage{caption, subcaption, booktabs, float, multirow}
\usepackage{fullpage}
\usepackage{algorithm, algorithmic}
\usepackage{enumerate}
\usepackage{MnSymbol}
\usepackage{lipsum}
\usepackage{bm}

\ifpdf
  \DeclareGraphicsExtensions{.eps,.pdf,.png,.jpg}
\else
  \DeclareGraphicsExtensions{.eps}
\fi

\newcommand{\dif}{\mathrm{d}}

\newcommand{\mr}{\mathbb{R}}
\newcommand{\mc}{\mathbb{C}}

\newcommand{\diag}{diag}

\newcommand{\sech}{sech}

\newcommand{\bff}{\mathbf{f}}
\newcommand{\bfF}{\mathbf{F}}
\newcommand{\vphi}{\varphi}
\newcommand{\tensor}[1]{\mathbf{#1}}
\newcommand{\vphin}{\varphi^N}
\newcommand{\bvphin}{\bm{{\varphi}}^N}

\newtheorem{theorem}{Theorem}[section]

\newtheorem{definition}[theorem]{Definition}

\newtheorem{lemma}[theorem]{Lemma}

\crefname{hypothesis}{Hypothesis}{Hypotheses}

\title{The Sparse-Grid-Based Adaptive Spectral Koopman Method}

\author{
Bian Li\thanks{Department of Industrial and Systems Engineering, Lehigh University, Bethlehem, PA 
  (\href{mailto:bil215@lehigh.edu}{bil215@lehigh.edu}, \href{mailto:xiy518@lehigh.edu}{xiy518@lehigh.edu})}
\and 
Yue Yu\thanks{Department of Mathematics, Lehigh University, Bethlehem, PA   
  (\href{mailto:yuy214@lehigh.edu}{yuy214@lehigh.edu})}
\and 
Xiu Yang\footnotemark[1]
}

\begin{document}

\maketitle

\begin{abstract}
The adaptive spectral Koopman (ASK) method was introduced to numerically solve autonomous dynamical systems that lay the foundation of numerous applications across different fields in science and engineering. 
Although ASK achieves high accuracy, it is computationally more expensive for multi-dimensional systems compared with conventional time integration schemes like Runge-Kutta. 
In this work, we combine the sparse grid and ASK to accelerate the computation for multi-dimensional systems.
This sparse-grid-based ASK (SASK) method uses the Smolyak structure to construct multi-dimensional collocation points as well as associated polynomials that are used to approximate eigenfunctions of the Koopman operator of the system.
In this way, the number of collocation points is reduced compared with using the tensor product rule.  
We demonstrate that SASK can be used to solve partial differential equations based-on their semi-discrete forms.
Numerical experiments illustrate that SASK balances the accuracy with the computational cost, and hence accelerates ASK. 
\end{abstract}

\textbf{Keywords:  } \textit{dynamical systems, sparse grids, Koopman operator, spectral-collocation method, partial differential equations}

\section{Introduction}

The Koopman operator~\cite{koopman1931hamiltonian} is an infinite-dimensional \emph{linear} operator that describes the evolution of a set of observables.
It provides a principled and often global framework to describe the dynamics of a finite-dimensional \emph{nonlinear} system. 
Consequently, the Koopman operator approach to \emph{nonlinear} dynamical systems has attracted considerable attention in recent years. 
One can define its eigenvalues, eigenfunctions, and modes, and then use them to represent dynamically interpretable low-dimensional embeddings of high-dimensional state spaces to construct solutions through linear superposition~\cite{brunton2021modern}. 
In particular, the spectrum of the Koopman operator in properly defined spaces does not contain continuous spectra, and the observable of the system can be represented as a linear combination of eigenfunctions associated with
\emph{discrete} eigenvalues of the Koopman operator~\cite{mezic2020spectrum,korda2020data,nakao2020spectral}. 
   
The Koopman operator provides powerful analytic tools to understand behaviors of dynamical systems by conducting Koopman mode analysis.
Such analysis starts with a choice of a set of linearly independent observables, and the Koopman operator is then analyzed through its action on the subspace spanned by the chosen observables~\cite{mezic2005spectral}.
This approach has been applied to study ordinary differential equations (ODEs), partial differential equations (PDEs)~\cite{wilson2016isostable,nathan2018applied,page2018koopman,nakao2020spectral},  disspative dynamical systems~\cite{mezic2020spectrum}, etc.
Furthermore, novel numerical schemes, especially data-driven algorithms, motivated by or related to the Koopman operator have attracted much attention in the past decade.
For example, the {\em dynamic mode decomposition} (DMD)~\cite{rowley2009spectral, schmid2010dynamic, tu2014dynamic,proctor2016dynamic, kutz2016multiresolution, nathan2018applied,askham2018variable} and its variants like extended DMD (EDMD)~\cite{williams2015data} use snapshots of a dynamical system to extract temporal features as well as correlated spatial activity.
Subsequently, they can predict the behavior of the system in a short time.
These approaches have been applied to design filters
(e.g.,~\cite{surana2016linear, netto2018robust}), train neural networks
(e.g.,~\cite{dogra2020optimizing}), etc.

In~\cite{li2022adaptive} we propose a novel numerical method based on the spectral-collocation method (i.e., the pseudospectral method)~\cite{fornberg1998practical,trefethen2000spectral} to implement the Koopman operator approach to solving nonlinear ordinary differential equations. 
This method leverages the differentiation matrix in spectral methods to approximate the generator of the Koopman operator, and then conducts eigen-decomposition numerically to obtain eigenvalues and eigenvectors that approximate Koopman operator's eigenvalues and eigenfunctions, respectively.
Here, each element of an eigenvector is the approximation of the associated eigenfunction evaluated at a collocation point.
The Koopman modes are approximated by computing eigenvalues, eigenvectors, and the observable based on the initial state.
This approach is more efficient than the conventional ODE solvers such as Runge-Kutta and Adam-Bashforth for low-dimensional ODEs in terms of computational time , especially when evaluating the dynamics is costly~\cite{li2022adaptive}.
This is because it allows evaluating the dynamic of the system at multiple collocation points simultaneously instead of computing them sequentially at different time steps as the aforementioned state-of-the-art ODE solvers.
In other words, ASK introduces a new parallelization mechanism for solving ODEs, which makes it more efficient than conventional approaches for some problems.

However, ASK's efficiency decreases as the system's dimension increases as it employs the tensor product rule to construct multi-dimensional collocation points and basis functions for polynomial interpolation. Consequently, the number of such points as well as basis functions increases exponentially.
This number is associated with the size of the eigen-decomposition problem and the linear system in the ASK scheme.
Therefore, ASK is less efficient in multi-dimensional cases.
To overcome this difficulty, we propose to combine the sparse grids method with ASK, wherein the Smolyak structure is applied to construct collocations points.
This sparse-grids-based ASK (SASK) method reduces the number of collocation points and that of basis functions used in the vanilla ASK. 
Hence, the computational efficiency is enhanced. In numerical experiments, we demonstrate that SASK can solve PDEs accurately based on their semi-discrete forms which are high-dimensional ODE systems.


The paper is organized as follows. \Cref{sec:background} introduces the background topics. A detailed discussion of  the sparse-grid-based adaptive spectral Koopman method follows in \Cref{sec:SASK}. We then show our numerical results in \Cref{sec:SASK_experiments}. Finally, \Cref{sec:dicussion_conclusion} concludes the paper with a summary and further discussion.


\section{Background}
\label{sec:background}

\subsection{Koopman operator} 
\label{subsec:koopman}
Borrowing notions from~\cite{kutz2016dynamic}, we consider an autonomous system described by the ordinary differential equations
\begin{equation}
  \label{eq:auto}
  \frac{\dif \bm x}{\dif t} = \bff(\bm x),
\end{equation}
where the state $\bm x=(x_1,x_2,\dotsc,x_d)^\top$ belongs to an $d$-dimensional smooth manifold $\mathscr{M}$, and the dynamics $\bff: \mathscr{M} \rightarrow \mathscr{M}$ does not explicitly depend on time $t$. 
Here, $\bff$ is a possibly nonlinear vector-valued smooth function, of the same dimension as $\bm x$.
In many studies, we aim to investigate the behavior of observables on the state space.
For this purpose, we define an observable to be a scalar function $g:\mathscr{M}\rightarrow\mathbb{R}$, where $g$ is an element of some function space $\mathcal{G}$ (e.g., $\mathcal{G}=L^2(\mathscr{M})$ as in~\cite{mezic2005spectral}). 
The flow map $\bfF_t: \mathscr{M} \rightarrow \mathscr{M}$ induced by the dynamical system~\eqref{eq:auto} depicts the evolution of the system as
\begin{equation}
\label{eqn:flow}
    \bm x(t_0 + t) = \bfF_t(\bm x(t_0)) = \bm x(t_0) + \int_{t_0}^{t_0 + t} \bff(\bm x(s)) \, d s. 
\end{equation}
Now we define the Koopman operator for continuous-time dynamical systems as follows~\cite{mezic2020spectrum}:
\begin{definition}
  Consider a family of operators $\{\mathcal{K}_t\}_{t\geq 0}$ acting on the space of observables so that
  \[\mathcal{K}_t g(\bm x_0) = g(\mathbf{F}_t(\bm x_0)),\] where $\bm x_0 = \bm x(t_0)$. 
  We call the family of operators $\mathcal{K}_t$ indexed by time t the Koopman
  operators of the continuous-time system~\eqref{eq:auto}.
\end{definition}
By definition, $\mathcal{K}_t$ is a \emph{linear} operator acting on the
function space $\mathcal{G}$ for each fixed $t$. Moreover, $\{\mathcal{K}_t\}$
form a semi-group.



\subsection{Infinitesimal generator} 
\label{subsec:generator}
The Koopman spectral theory~\cite{mezic2005spectral, rowley2009spectral} 
unveils properties that enable the Koopman operator to convert nonlinear finite-dimensional dynamics into linear infinite-dimensional dynamics. 
A key component in such spectral analysis is the infinitesimal generator (or generator for brievity) of the Koopman operator.
Specifically, 
for any smooth obervable function $g$, the generator of the Koopman operator $\mathcal{K}_t$, denoted as $\mathcal{K}$, is given by
\begin{align}
  \mathcal{K} g = \lim_{t \rightarrow 0} \frac{\mathcal{K}_t g
  -g}{t}, \label{eqn:Koopman_def_generator}
\end{align}
which leads to
\begin{align}
  \label{eqn:Koopman_derivative_relation}
  \mathcal{K}g(\bm x) = \nabla g(\bm x) \cdot \frac{\dif \bm
  x}{\dif t} = \frac{\dif g(\bm x)}{\dif t}. 
\end{align}
Denoting $\vphi$ an eigenfunction of $\mathcal{K}$ and $\lambda$ the eigenvalue associated with $\vphi$, we have
  $\mathcal{K} \vphi(\bm x) = \lambda \vphi(\bm x)$, 
 and hence
$ \lambda \vphi(\bm x) = \mathcal{K} \vphi(\bm x) = \frac{\dif \vphi(\bm x)}{\dif t}.$ 
This indicates that $\varphi(\bm x(t_0+t))=e^{\lambda t}\varphi(\bm x(t_0))$, i.e., 
\begin{equation}
  \mathcal{K}_t \varphi(\bm x(t_0)) = e^{\lambda t} \varphi(\bm x(t_0)).
\end{equation}
Therefor, $\varphi$ is an eigenfunction of $\mathcal{K}_t$ associated with eigenvalue $\lambda$. 
Of note, following notations in literatur, we consider the eigenpair for $\mathcal{K}_t$ as $(\varphi, \lambda)$ instead of $(\varphi, e^{\lambda t})$.

Now suppose $g$ exists in the function space spanned by all the eigenfunctions $\vphi_j$ (associated with
eigenvalues $\lambda_j$) of $\mathcal{K}$,
i.e., $g(\bm x) = \sum_j^\infty c_j \vphi_j(\bm x)$, then
\begin{align}
  \mathcal{K}_t [g(\bm x(t_0))] = \mathcal{K}_t \left[ \sum_j^\infty c_j
  \vphi_j(\bm x(t_0)) \right]
  = \sum_j^\infty c_j \mathcal{K}_t[\vphi_j(\bm x(t_0))] 
    \label{eqn:koopman_solution}.
\end{align}
Hence, 
\begin{equation}
  \label{eqn:inf_expansion}
  g(\bm x(t_0+t)) = \sum_j^\infty c_j \vphi_j(\bm x(t_0)) e^{\lambda_j t}. 
\end{equation}
Similarly, for a vector-valued observable $\bm g: \mathscr{M} \rightarrow \mr^d$ with $\bm g:=(g_1(\bm x), g_2(\bm x), \dotsc, g_d(\bm x))^\top$, the system of observables becomes 
\begin{align}
    \frac{\dif \bm g(\bm x)}{\dif t} = \mathcal{K} \bm g(\bm x)
    = 
    \begin{bmatrix}
        \mathcal{K} g_1(\bm x) \\
        \mathcal{K} g_2(\bm x) \\
        \vdots \\
        \mathcal{K} g_d(\bm x) \\
    \end{bmatrix}
    = \sum_j^\infty \lambda_j \vphi_j(\bm x) \bm c_j, \label{eqn:Koopman_system_observable}
\end{align}
where $\bm c_j \in \mc^d$ is called the $j$th Koopman mode with
$\bm c_j := (c_j^1, c_j^2, \dotsc, c_j^d)^\top$. 

The ASK method uses the following truncated form of~\Cref{eqn:inf_expansion}
\begin{equation}
  \label{eqn:finite_expansion}
  g(\bm x(t_0+t)) = \sum_j^\infty c_j \vphi_j(\bm x(t_0)) e^{\lambda_j t}\approx \sum_{j=0}^N \tilde c_j \vphi_j^N(\bm x(t_0)) e^{\tilde\lambda_j t}
\end{equation}
for $d=1$. Here, $\varphi_j$ are approximated by $N$-th order interpolation polynomials $\varphi_j^N$, where $N$ is a positive integer. Also, $\lambda_j$ and $c_j$ are approximated by $\tilde\lambda_j$ and $\tilde c_j$~\cite{li2022adaptive}. For $d>1$, $\varphi_j^N$ is constructed by tensor product rule with one-dimensional interpolation polynomials.



\section{Sparse-Grid-Based Adaptive Spectral Koopman Method}
\label{sec:SASK}

As mentioned in~\cite{li2022adaptive}, ASK suffers from the curse of dimensionality as we approach high-dimensional systems. It is not surprising to see this phenomenon in numerical integration and interpolation on multidimensional domains when the tensor product rule is used to construct high-dimensional quadrature points. 
Specifically, if $N$ denotes the number of points for one dimension and $d$ denotes the number of dimensions, the tensor product rule gives a domain containing $N^d$ points, which quickly grows prohibitive with $d$.

The sparse grid method is one of the most effective approaches to overcome the aforementioned difficulties to a certain extent as it needs significantly fewer points in the computation.
This method is also known as the Smolyak grid (or Smolyak's construction) in the name of Sergei A. Smolyak~\cite{smolyak1963quadrature}. 
A series of seminal works further studied the properties of the sparse grid method and completed the framework~\cite{bungartz2004sparse,griebel1990combination,gerstner1998numerical,zenger1991sparse}. 
The full $N^d$-grid is a direct consequence of the tensor product of the points in each dimension, while the sparse grid method chooses only a subset of these grid points so that the total number increases much slower in $d$. 
As shown by Zenger~\cite{zenger1991sparse}, the total number of points is polynomial in $d$. 
This drastically reduces the computation complexity, enabling a more efficient variant of ASK.
In this section, we introduce the sparse-grid-based adaptive spectral Koopman (SASK) method, an accelerated version of ASK. 

\subsection{Sparse grids for interpolation}
\label{subsec:smolyak_grid}

The idea of sparse grids is that some grid points contribute more than the others in the numerical approximation. Thus, it does not undermine the interpolation if only a subset of the important grid points are utilized. 
In fact, the order of the error only increases slightly~\cite{zenger1991sparse}. 
The polynomial interpolation in SASK borrows the ideas from~\cite{judd2014smolyak} which leveraged Chebyshev extreme points to generate the sparse grids and Chebyshev polynomials to construct the basis functions. 
Following a similar structure, this subsection first discusses the generation of the sparse grid. 
Then, the basis function interpolation is explained, followed by the computation of the coefficients in the linear combination of the basis functions. 

\subsubsection{Sparse grid construction}
\label{subsubsec:sparse_grid_construction}
In this work, the points of a sparse grid are based on the extreme points of the Chebyshev polynomials. 
Specifically, denote $\xi_j = \cos\big(\frac{j \pi}{n-1}\big)$ for $j \in \{0,1,\dotsc,n-1\}$. 
The construction of the sparse grids in a multi-dimensional domain builds on the uni-dimensional set of Chebyshev points that satisfy the Smolyak rule. 

Let $\{\mathcal{N}_i\}_{i \geq 1}$ be a sequence of sets which contain the Chebyshev points such that the number of points in set $i$ is $m(i) = 2^{i - 1} + 1$ for $i \geq 2$ and $m(1) = 1$, and that $\mathcal{N}_i \subset \mathcal{N}_{i+1}$. Then,
\begin{align*}
    \mathcal{N}_1 &= \{ 0 \}, \quad
    \mathcal{N}_2 = \{ 0, -1, 1 \}, \quad
    \mathcal{N}_3 = \left\{ 0, -1, 1, - \frac{\sqrt{2}}{2}, \frac{\sqrt{2}}{2} \right\}. 
\end{align*}
To construct the sparse grid, we need another parameter, the approximation level $\kappa \in \mathbb{Z}_{\geq 0}$. This parameter controls the number of points $n$ in one dimension, thus further dictating the overall degree of approximation. In particular, $n = 2^\kappa + 1$ when $\kappa \geq 1$ and $n = 1$ when $\kappa = 0$. Let $i_j$ denote the index of the set in dimension $j$. Then, the Smolyak rule states that
\begin{align*}
    d \leq \sum_{j = 1}^d i_j \leq d + \kappa.
\end{align*}
Here, we show an example with $d = 2, \kappa = 2$. In this case, $2 \leq i_1 + i_2 \leq 4$, and hence the possible combinations are as follows:
\begin{equation}
\begin{aligned}
    &1) \enspace i_1 = 1, i_2 = 1, \quad 2) \enspace i_1 = 2, i_2 = 1, \quad 3) \enspace i_1 = 3, i_2 = 1, \\
    &4) \enspace i_1 = 1, i_2 = 2, \quad 5) \enspace i_1 = 1, i_2 = 3, \quad 6) \enspace i_1 = 2, i_2 = 2.
\end{aligned}
  \label{eq:indx_comb}
\end{equation}
Let $\mathcal{S}(\cdot, \cdot)$ be the tensor product of two sets of points.
Then, the combinations $(i_1, i_2)$ in~\Cref{eq:indx_comb} provide  $\mathcal{S}(\mathcal{N}_{i_1}, \mathcal{N}_{i_2})$. 
For example, $\mathcal{S}(\mathcal{N}_1, \mathcal{N}_3) = \left\{ (0, 0), (0, -1), (0, 1), (0, -\frac{\sqrt{2}}{2}), (0, \frac{\sqrt{2}}{2}) \right\}$.
In this way, the sparse grid can be constructed as the union of $\mathcal{S}(\mathcal{N}_{i_1}, \mathcal{N}_{i_2})$. 
The illustration of the sparse grids and its comparison with the full grids are enclosed in the appendix ~\Cref{append:sparse}.

By construction, there are repetitions of points in the union of
$\mathcal{S}(\mathcal{N}_{i_1}, \mathcal{N}_{i_2})$ since the sets are nested.
For example, $\mathcal{S}(\mathcal{N}_1, \mathcal{N}_2)\cap \mathcal{S}(\mathcal{N}_1, \mathcal{N}_3)=\mathcal{S}(\mathcal{N}_1, \mathcal{N}_2)$.
Hence, a more concise way to construct sparse grids is to apply disjoint sets~\cite{gerstner1998numerical,judd2014smolyak}.
Denote $\mathcal{A}_i := \mathcal{N}_{i} \backslash \mathcal{N}_{i-1}$ for $i = 2,3,\dotsc$ and $\mathcal{A}_1 := \mathcal{N}_1$. Then, we have $\mathcal{A}_i \cap \mathcal{A}_{j \neq i} = \emptyset$. The number of points in $\mathcal{A}_i$ is computed by $\bar{m}(i) = m(i) - m(i-1) = 2^{i-2}$ for $i \geq 3$, $\bar{m}(1) = 1$, and $\bar{m}(2) = 2$. Specifically,
\begin{align*}
    \mathcal{A}_1 &= \{ 0 \}, \quad
    \mathcal{A}_2 = \{-1, 1 \}, \quad 
    \mathcal{A}_3 = \left\{ - \frac{\sqrt{2}}{2}, \frac{\sqrt{2}}{2} \right\}. 
\end{align*}
Subsequently, we use the union of $\mathcal{S}(\mathcal{A}_{i_1}, \mathcal{A}_{i_2})$ to construct sparse grids. 
By construction, $ \bigcup_{i_1, i_2}\mathcal{S}(\mathcal{N}_{i_1}, \mathcal{N}_{i_2}) = \bigcup_{i_1, i_2}\mathcal{S}(\mathcal{A}_{i_1}, \mathcal{A}_{i_2})$.

\subsubsection{Polynomial interpolation}
We aim to approximate a smooth multivariate function $h(\bm x)$ with a linear combination of polynomials that serve as the basis functions. 
Here, we choose the Chebyshev polynomials of the first kind to be the univariate basis functions. It follows to construct multivariate basis functions with the tensor product of the univariate basis functions. 
The Chebyshev polynomials of the first kind $T_n$ are given by a recurrence relation: $T_{n+1}(x) = 2x T_n(x) - T_{n-1}(x)$ with $T_0(x) = 1, T_1(x) = x$ and $x \in [-1,1]$. Hence, the basis functions used are $\psi_1(x) = 1, \psi_2(x) = x, \psi_3(x) = 2x^2 - 1, \psi_4(x) = 4x^3 - 3x$, and so on. Corresponding to the sets $\{ \mathcal{A}_i \}$, we define the disjoint sets of uni-dimensional basis functions $\{ \mathcal{F}_i \}$ by
\begin{align*}
    \mathcal{F}_1 &= \{ \psi_1(x) \}, \quad
    \mathcal{F}_2 = \{\psi_2(x), \psi_3(x) \}, \quad 
    \mathcal{F}_3 = \left\{ \psi_4(x), \psi_5(x) \right\}.
\end{align*}
Let $(x_1, x_2) \in [-1, 1] \times [-1, 1]$. By the Smolyak rule, the example above has the following basis function tensor products,
\begin{align*}
    &1)\quad  \mathcal{S}(\mathcal{F}_1, \mathcal{F}_1) = 
            \{ \psi_1(x_1) \psi_1(x_2) \}, \\ 
    &2)\quad  \mathcal{S}(\mathcal{F}_2, \mathcal{F}_1) =
            \{ \psi_2(x_1) \psi_1(x_2), \psi_3(x_1) \psi_1(x_2)\}, \\ 
    &3)\quad  \mathcal{S}(\mathcal{F}_3, \mathcal{F}_1) =
            \{ \psi_4(x_1) \psi_1(x_2), \psi_5(x_1) \psi_1(x_2)\}, \\
    &4)\quad   \mathcal{S}(\mathcal{F}_1, \mathcal{F}_2) = 
            \{ \psi_1(x_1)\psi_2(x_2), \psi_1(x_1)\psi_3(x_2) \}, \\ 
    &5)\quad  \mathcal{S}(\mathcal{F}_1, \mathcal{F}_3) = 
            \{ \psi_1(x_1) \psi_4(x_2), \psi_1(x_1) \psi_5(x_2)\}, \\
    &6)\quad \mathcal{S}(\mathcal{F}_2, \mathcal{F}_2) = 
            \{ \psi_2(x_1)\psi_2(x_2), \psi_2(x_1)\psi_3(x_2),
                \psi_3(x_1)\psi_2(x_2), \psi_3(x_1)\psi_3(x_2) \}.
\end{align*}
The union of these $\mathcal{S}(\mathcal{F}_{i_1}, \mathcal{F}_{i_2})$ forms the
basis functions for the polynomial interpolation.

Suppose the total number of sparse grid points is $N$. Then, the total number of
basis functions is also $N$, and we denote them as $\Psi_{l}(\bm x)$ by ordering them with index $l$.
For example, $\Psi_1(\bm x) = \psi_1(x_1) \psi_1(x_2), \Psi_2(\bm x) = \psi_2(x_1) \psi_1(x_2)$ and so on. 
It then remains to approximate $h(\bm x)$ as the following linear combination,
\begin{align*}
    h(\bm x) \approx h^N(\bm x) = \sum_{l=1}^N w_l \Psi_l(\bm x), 
\end{align*}
where $w_l$ denote the unknown coefficients. 
Given the grid points $\{ \bm{\xi}_l \}_{l=1}^N$, we can write 
\begin{align}
    \begin{bmatrix}
        h^N(\bm \xi_1) \\
        h^N(\bm \xi_2) \\
        \vdots \\
        h^N(\bm \xi_N)
    \end{bmatrix}
    = 
    \begin{bmatrix}
        \Psi_1(\bm \xi_1) & \Psi_2(\bm \xi_1) & \dots & \Psi_N(\bm \xi_1) \\
        \Psi_1(\bm \xi_2) & \Psi_2(\bm \xi_2) & \dots & \Psi_N(\bm \xi_2) \\
        \vdots & \vdots & \ddots & \vdots \\
        \Psi_1(\bm \xi_N) & \Psi_2(\bm \xi_N) & \dots & \Psi_N(\bm \xi_N) \\
    \end{bmatrix}
    \begin{bmatrix}
        w_1 \\
        w_2 \\
        \vdots \\
        w_N 
    \end{bmatrix} \label{eqn:sparse_grid_interpolation}
\end{align}
as $\bm h^N = \tensor M \bm w$, where $\bm h^N = \left(h^N(\bm\xi_1),\dotsc, h^N(\bm\xi_N)\right)^\top, \bm w = (w_1,\dotsc, w_N)^\top$, and $M_{ij}=\Psi_j(\bm\xi_i)$. 
Here, matrix $\tensor M$ is full-ranked due to the orthogonality of Chebyshev polynomials. 
Vector $\bm w$ can be obtained by $\bm w = \tensor M^{-1} \bm h^N$ when $\bm h^N$ and $\tensor M$ are known.

\subsection{Finite-dimensional approximation}
\label{subsec:SASK_finite_approximation}
To leverage the properties of the Koopman operator for solving dynamical systems, we intend to find the approximation of~\Cref{eqn:inf_expansion} as
\begin{align}
    \label{eqn:trun_expansion}
    g(\bm x(t + t_0)) \approx g_N(\bm x(t + t_0)) = \sum_{j=1}^{N} \tilde c_j \vphi_j^N(\bm x(t_0)) e^{\tilde \lambda_j t},
\end{align}
where $\tilde c_j$ is the approximate Koopman mode, $\tilde \lambda_j$ is the approximate eigenvalue, and $\vphi_j^N$ is the polynomial approximation of eigenfunction $\varphi_j$.
Without loss of generality, we will assume that $t_0 = 0$ and denote $\bm x_0 := \bm x(t_0) = \bm x(0)$. 
The property of the infinitesimal generator~\Cref{eqn:Koopman_derivative_relation} leads to $\mathcal{K} \vphi(\bm x) = \frac{\dif \bm x}{\dif t} \cdot \nabla \vphi(\bm x)$ for any eigenfunction $\varphi$. 
Since $\frac{\dif \bm x}{\dif t} = \bff(\bm x)$, we have
\begin{align}
    \mathcal{K} \vphi = \bff \cdot \nabla \vphi = f_1 \frac{\partial  \vphi}{\partial x_1} + f_2\frac{\partial \vphi}{\partial x_2} + \dotsc + f_d\frac{\partial \vphi}{\partial x_d}.
    \label{eqn:generator_expansion}
\end{align}
The polynomial approximation $\vphin_j$ in~\Cref{eqn:trun_expansion} can be
obtained based on~\Cref{eqn:generator_expansion}.

Consider the following polynomial approximation of an eigenfunction
\begin{align*}
    \vphi(\bm x) \approx \vphin(\bm x) = \sum_{l=1}^N w_l \Psi_l(\bm x).
\end{align*}
It then follows that 
\begin{align*}
    \frac{\partial \vphi(\bm x)}{\partial x_i} \approx \frac{\partial \vphin(\bm x)}{\partial x_i} = \sum_{l=1}^N w_l \frac{\partial \Psi_l(\bm x)}{\partial x_i} \quad \forall i.
\end{align*}
Denote the sparse grid points in $\mathbb{R}^d$ by $\{ \bm \xi_l \}_{l=1}^N$, where $\bm \xi_l := (\xi_{l_1}, \xi_{l_2}, \dotsc, \xi_{l_d}) \in \mr^d$. 
Replacing $h$ in~\Cref{eqn:sparse_grid_interpolation} with $\varphi$, we have $\bvphin = \tensor M \bm w$, where $\bvphin$ is the vector of $\varphi$ evaluated at $\bm \xi_l$.
Accordingly, let matrix $\tensor G_i$ be $\partial_{x_i}\Psi_l(\bm x)$
evaluated at the sparse grids points, we have
\begin{align*}
    \tensor G_i = 
    \begin{bmatrix}
         \frac{\partial \Psi_1}{\partial x_i} (\bm \xi_1) & \frac{\partial \Psi_2}{\partial x_i} (\bm \xi_1) & \dots & \frac{\partial \Psi_N}{\partial x_i} (\bm \xi_1) \\
        \frac{\partial \Psi_1}{\partial x_i} (\bm \xi_2) & \frac{\partial \Psi_2}{\partial x_i} (\bm \xi_2) & \dots & \frac{\partial \Psi_N}{\partial x_i} (\bm \xi_2) \\
        \vdots & \vdots & \ddots & \vdots \\
        \frac{\partial \Psi_1}{\partial x_i} (\bm \xi_N) & \frac{\partial \Psi_2}{\partial x_i} (\bm \xi_N) & \dots & \frac{\partial \Psi_N}{\partial x_i} (\bm \xi_N) \\
    \end{bmatrix}.
\end{align*}
Then, $\partial_{x_i} \bvphin = \tensor G_i \bm w$, where $(\partial_{x_i} \vphin)_l := \frac{\partial \Psi}{\partial x_i}(\bm \xi_l)$. 
Let $\tensor K \in \mr^{N \times N}$ be the finite-dimensional approximation of $\mathcal{K}$. Given the dynamics $\bff = [f_1, f_2,\dotsc, f_d]^\top$, \Cref{eqn:generator_expansion} implies
\begin{align}
    \label{eqn:SASK_finite_approximation}
    \tensor K \bvphin = \sum_i^d \: \diag{\big( f_i(\bm \xi_1), \dotsc, f_i(\bm \xi_N) \big)} \: \tensor G_i \bm w. 
\end{align}

\subsection{Eigen-decomposition}
\label{subsec:SASK_eigen_decomp}
With the discretized Koopman operator, we intend to obtain the eigenfunction values using the eigen-decomposition. 
One can formulate the eigenvalue problem $\mathcal{K} \vphi_j = \lambda_j \vphi_j$, where $(\vphi_j, \lambda_j)$ is an eigenpair of $\mathcal{K}$. 
Correspondingly, the discrete eigenvalue problem is $\tensor K \bvphin_j = \tilde \lambda_j \bvphin_j$. 
By ~\Cref{eqn:SASK_finite_approximation} and $\bvphin_j = \tensor M \bm w_j$, we have 
\begin{align}
    \sum_i^d \: \diag{\big( f_i(\bm \xi_1), \dotsc, f_i(\bm \xi_N) \big)} \: \tensor G_i \bm w_j = \tilde \lambda_j \tensor M \bm w_j. \label{eqn:SASK_generalized_eigen_problem}
\end{align}
Let $\tensor U := \sum_i^d \: \diag{\big( f_i(\bm \xi_1), \dotsc, f_i(\bm \xi_N) \big)} \: \tensor G_i$, $\tensor U \bm w_j = \tilde \lambda_j \tensor M \bm w_j$ is a generalized eigenvalue problem, from which we solve for $\bm w_j$. 
For compactness, we write this in the matrix form 
\begin{equation}
  \label{eqn:generalized_eig}
\tensor U \tensor W = \tensor M \tensor W \tensor \Lambda, 
\end{equation}
where
$ \tensor W := 
    \begin{bmatrix}
        \bm w_1 & \bm w_2 & \dots & \bm w_N
    \end{bmatrix}, 
    \tensor \Lambda := \text{diag}\left(\tilde \lambda_1, \tilde \lambda_2,
    \dotsc, \tilde\lambda_N\right)$.
Then, the matrix of eigenfunctions can be defined by $\tensor \Phi^N := \tensor M
\tensor W$, whose $j$th column is $\bvphin_j$. 

We note that SASK requires solving a generalized eigenvalue problem while ASK uses a standard eigen-decomposition.
This is because matrix $\tensor M$ is identity matrix $\tensor I$ in ASK as it uses Lagrange polynomial for the interpolation which indicates that $\Psi_i(\bm\xi_j)=\delta_{ij}$, where $\delta_{ij}$ is the Kronecker delta function. 
Thus, in this setting, the generalized eigenvalue problem~\Cref{eqn:generalized_eig} is degenerated as $\tensor U\tensor W=\tensor W\tensor\Lambda$.
Further, by construction, $\tensor G_i=\tensor D_i \tensor M$, where $\tensor D_i$ is the differentiation matrix in the $i$th direction. 
This formula reduces to $\tensor G_i=\tensor D_i$ when $\tensor M=\tensor I$ in ASK (see~\cite{li2022adaptive}).
On the other hand, SASK uses a more general setting for the interpolation, i.e., $\Psi_i$ are not necessarily Lagrange polynomials. 
Therefore, $\tensor M\neq \tensor I$ and the differentiation matrices $\tensor D_i$ need to be obtained by solving a linear system.
Instead of computing $\tensor D_i$ explicitly, we compute $\tensor G_i$ in SASK.


\subsection{Constructing the solution}
\label{subsec:SASK_solution_construction}
The eigen-decomposition yields eigenfunction values $\vphin_j$ at the sparse grid points $\bm \xi_l$. 
By construction,
the central point of the domain is also the first sparse grid point generated. 
For example, $(0,0,\dotsc,0) = \bm \xi_1$ for a multi-dimensional domain $[-1, 1]^d$. 
Hence, to avoid interpolating the eigenfunction when $\bm x_0 \notin \{\bm \xi_l\}$, we propose to construct a neighborhood of $\bm x_0$ defined by $[\bm x_0 - \bm r, \bm x_0 + \bm r]$, where $\bm r = (r_1, r_2, \dotsc, r_d)^\top$ is the radius. 
Equivalently, the neighborhood in dimension $i$ is $[x_0^i - r_i, x_0^i + r_i]$. 
For simplicity, we apply the isotropic setting with $r := r_1 = r_2 = \dotsc = r_d$ in this work, but we emphasize that it is not necessary, and that the anisotropic setting might be more effective. 
Therefore, the observable of the dynamical system is constructed as
\begin{align}
    \label{eqn:obs}
    g_N(\bm x(t)) = \sum_{j = 1}^N \tilde c_j \vphin_j(\bm x_0) e^{\tilde \lambda_j t}. 
\end{align}
Setting $t = 0$, we compute the approximate Koopman modes $\tilde c_j$ using the following equation,
\begin{align*}
    g(\bm x_0) \approx g_N(\bm x_0) = \sum_{j = 1}^N \tilde c_j \vphi_j^N(\bm x_0),
\end{align*}
which must be satisfied for different initial conditions in the neighborhood of $\bm x_0$.
Thus, by considering all sparse grid points as different initial conditions, we have 
\begin{align*}
    g(\bm \xi_l) \approx g_N(\bm \xi_l) = \sum_{j = 1}^N \tilde c_j \vphi_j^N(\bm \xi_l), \quad l = 1, 2, \dotsc, N.
\end{align*}
These formulas can be summarized in a matrix form by defining the matrix of the sparse grid as 
\begin{align*}
    \tensor\Xi := 
    \begin{bmatrix}
        (\bm \xi_1)^\top \\
        (\bm \xi_2)^\top \\
        \vdots \\
        (\bm \xi_N)^\top
    \end{bmatrix}
    =
    \begin{bmatrix}
        \xi_{11} & \xi_{12} & \dotsc & \xi_{1d} \\
        \xi_{21} & \xi_{22} & \dotsc & \xi_{2d} \\
        \vdots & \vdots & \ddots & \vdots \\
        \xi_{N1} & \xi_{N2} & \dotsc & \xi_{Nd} \\
    \end{bmatrix},
\end{align*}
and denoting column $i$ of the matrix by $\tensor \Xi_i$. If we choose the vector-valued observable $\bm g(\bm x) = \bm x$, then the Koopman modes must satisfy $\tensor \Phi^N \bm c_i = \tensor \Xi_i$ for all $i$. The Koopman modes are computed by solving these linear systems. In a more compact form, $\tensor \Phi^N \tensor C = \tensor \Xi$. In particular, $\bm c_i$ is column $i$ of the matrix $\tensor C = (c_{ji})$, containing the Koopman modes for dimension $i$.

Finally, $\bm\xi_1=\bm x_0$ by construction, and hence, $\vphi_j^N(\bm x_0)$ is the first element of vector $ \bvphin_j$, denoted by $(\bvphin_j)_1$.
Therefore, the solution of the dynamical system is constructed as
\begin{align}
    \label{eqn:sol}
    \bm x(t) = \sum_{j = 1}^N \tilde c_j \vphin_j(\bm x_0) e^{\tilde \lambda_j t} = \sum_{j = 1}^N \tilde c_j ( \bvphin_j)_1 e^{\tilde \lambda_j t},
\end{align}
wherein the observable function is identity. 

\subsection{Adaptivity}
Due to the finite-dimensional approximation of $\mathcal{K}$ and local approximation (in the neighborhood of $\bm x_0$) of $\varphi$, the accuracy of the solution decays as the system evolves in time. This is particularly the case for systems with highly nonlinear dynamics. 
To solve this problem, we adaptively update $\tensor \Phi^N$, $\tensor \Lambda$, and $\tensor C$ via procedures discussed in~\Cref{subsec:SASK_finite_approximation}--~\Cref{subsec:SASK_solution_construction}. 

Specifically, we set a series of check points in the time span $0 < \tau_1 < \tau_2 < \dotsc < \tau_n < T$. 
On each of the point, the algorithm examines whether the neighborhood of $\bm x(\tau_k)$ is ``valid'', so as to further guarantee the accuracy of the finite-dimensional approximation.
For such a purpose, we define the acceptable range 
\begin{align}
    \label{eqn:acceptable_range}
    R_i := \left[L_i + \gamma r, U_i - \gamma r \right],
\end{align}
where $L_i, U_i$ are the lower and upper bounds, $r$ is the radius mentioned in ~\Cref{subsec:SASK_solution_construction}, and $\gamma \in (0, 1]$ is a tunable parameter. At the initial time point, $L_i = x_0^i - r$ and $U_i = x_0^i + r$. For the current state $\bm x(\tau_k) = \left( x_1(\tau_k), x_2(\tau_k), \dotsc, x_d(\tau_k) \right)^\top$, the neighborhood $R_1 \times R_2 \times \dotsc \times R_d$ is valid if $x_i(\tau_k) \in R_i$ for all $i$. In the case where at least one component $x_i(\tau_k) \notin R_i$, we realize the update by the following procedures:
\begin{enumerate}
    \item Update $L_i = x_i(\tau_k) - r, \: U_i = x_i(\tau_k) + r$ for all $i$. \label{itm:update_1}
    
    \item Generate the sparse grid and compute matrices $\tensor M, \tensor G_i$. \label{itm:update_2}
    
    \item Apply the eigen-decomposition to update $\tensor \Phi^N, \tensor \Lambda$. \label{itm:update_3}
    
    \item Compute the Koopman modes $\tensor C$ with the updated $\tensor \Phi^N$. \label{itm:update_4}
    
     \item Construct solution $\bm x(t)$ by replacing $e^{\tilde \lambda_jt}$ with $e^{\tilde \lambda_j (t-\tau_k)}$ in~\Cref{eqn:sol}.
\end{enumerate}
Step 5 above comes from the adjustment $t_0 = \tau_k$ and $\bm x_0 = \bm x(t_0) = \bm x(\tau_k)$ whenever the update is performed. Notably, the parameter $\gamma$ controls the strictness of the validity check. 
When $\gamma$ is large, the updates occur more frequently. 
Setting $\gamma = 1$ is tantamount to forcing an update at every check point. 
As addressed in~\cite{li2022adaptive}, SASK also differs from traditional ODE solvers as it does not discretize the system in time, and the check points are essentially different from the time grid points in traditional solvers.
Instead, the discretization is in the state space. As a result, SASK is time-mesh-independent.

\subsection{Algorithm summary}
To leverage the properties of the Koopman operator, \Cref{subsec:SASK_finite_approximation}--~\Cref{subsec:SASK_solution_construction} finds the finite-dimensional approximation of the Koopman operator, and approximates the eigenfunctions and eigenvalues. 
The solution is obtained by a linear combination of the eigenfunctions. 
In order to preserve accuracy as time evolves, the adaptivity is added into the numerical scheme.
The complete algorithm is summarized in ~\Cref{algo:SASK_pseudo_code}.
\begin{algorithm}[htpb]
    \caption{Sparse-Grid-Based Adaptive Spectral Koopman Method}
    \label{algo:SASK_pseudo_code}
    \begin{algorithmic}[1]
        \REQUIRE $n, T, \bm x(0), r, \kappa, \gamma$
        \STATE{Set check points at $0 = \tau_0, \tau_1, \dotsc, \tau_n < T$.}
        \STATE{Let $L_i=x_0^i-r, U_i = x_0^i+r$ and set neighborhood $R_i=[L_i+\gamma r, U_i-\gamma r]$ for $i = 1,2,...,d$.}
        \STATE{Generate sparse grid points $\{\bm \xi_l \}_{l=1}^N$ and compute $\tensor M, \tensor G_i$ for $i = 1,2,\dotsc,d$.}
        \STATE{Apply eigen-decomposition to $\tensor U \tensor W = \tensor M \tensor W \tensor \Lambda$ and compute $\tensor \Phi^N = \tensor M \tensor W$.}
        \STATE{Solve linear system $\tensor \Phi^N \tensor C = \tensor \Xi$, where $\tensor \Xi$ is defined in ~\Cref{subsec:SASK_solution_construction}.}
        \FOR{$k= 1,2,3,\dotsc, n$}
            \STATE{Let $\nu_j$ be the first element of the $j$th column of $\tensor \Phi$. Construct solution at time $\tau_k$ as $\bm x(\tau_k) = \displaystyle\sum_j \tensor C(j, :) \nu_j e^{\tilde \lambda_j (\tau_k - \tau_{k-1})}$, where $\tensor C(j, :)$ is the $j$th row of $\tensor C$.}
            \IF{$x_i(\tau_k) \notin R_i$ for any $i$}
                \STATE{Set $L_i = x_i(\tau_k) - r, \: U_i = x_i(\tau_k) + r$ and $R_i = [L_i + \gamma r, U_i - \gamma r].$}
                \STATE{Repeat steps 3 - 5.}
            \ENDIF
        \ENDFOR
        \RETURN $\bm x(T) = \displaystyle\sum_{j} \tensor C(j,:)\nu_j e^{\tilde \lambda_j (T-\tau_{n})}.$
    \end{algorithmic}
\end{algorithm}
Of note, the construction of the sparse grid is based on the reference domain $[-1, 1]$ in practice. 
We then only need to re-scale the sparse grid points and matrices $\tensor
G_i$ on the reference domain by $\frac{U_i - L_i}{2} (\tensor \Xi_i + 1) + L_i$ and $\frac{2 \tensor G_i}{U_i - L_i}$ respectively, so that they match the real domain $[L_i, U_i]$. 


\section{Numerical Analysis}
In this section, we provide numerical analysis results to understand the
performance of the ASK and SASK method. We start with a lemma adapted
from~\cite{mezic2020spectrum}:
\begin{lemma}
Consider a linear dynamical system $\frac{\dif \bm x}{\dif t} =\tensor A\bm x$,
where $\bm x\in \mathbb{R}^d$ and $\tensor A\in \mathbb{R}^{d\times d}$. Then
the eigenvalues of $\tensor A$ are eigenvalues of the Koopman operator, and the
associated Koopman eigenfunctions are 
$\varphi_j(\bm x) = \langle \bm x, \bm w_j\rangle$ for $j=1,2,\dotsc,d$, where 
$\bm w_j$ are eigenvectors of $\tensor A^\top$ such that 
$\Vert w_j \Vert_2 = 1$ and $\langle\cdot, \cdot\rangle$
denotes the complex inner product on the manifold $\mathcal{M}$.
\end{lemma}
\begin{proof} Let $\lambda_j$ be an eigenvalue of $\tensor A$, then it is also
  an eigenvalue of $\tensor A^\top$. Assume 
  $\tensor A^\top\bm w_j =\lambda_j \bm w_j$, as shown in~\cite{mezic2020spectrum},
  we have
  \[ \frac{\dif \varphi_j}{\dif t}= \frac{\dif}{\dif t}\langle \bm x, \bm
  w_j\rangle =\langle \frac{\dif \bm x}{\dif t}, \bm
  w_j\rangle=\langle \tensor A\bm x, \bm w_j\rangle =\langle \bm x,\tensor
  A^\top\bm w_j\rangle =\lambda_j\langle \bm x, \bm w_j\rangle
  =\lambda_j\varphi_j.  \]
  Thus, $\varphi_j$ is an eigenfunction of the linear system's Koopman operator.
  Alternatively, using~\Cref{eqn:generator_expansion}, we have 
  \[\mathcal{K}\varphi_j = \sum_{i=1}^d (\tensor A\bm x)_i
  \dfrac{\partial\varphi_j}{\partial x_i}= \langle \tensor A\bm x, \bm w_j \rangle =
  \langle \bm x, \tensor A^\top\bm w_j \rangle = \lambda_j \langle \bm x, \bm w_j\rangle = \lambda_j\varphi_j.
\]
\end{proof}
These two different proofs of this lemma illustrate the connection of temporal 
derivatives and spatial derivatives via the Koopman operator.

Next, as long as $\tensor A$ has a full set of eigenvectors at distinct
eigenvalues $\lambda_j$, we have 
\[  g(\bm x(t_0+t)) = \sum_{j=1}^d c_j \vphi_j(\bm x(t_0)) e^{\lambda_j t} \]
for observable $g:\mathcal{M}\rightarrow \mathbb{R}$, and $g(\bm x(t_0)) =
\sum_{j=1}^d c_j \vphi_j(\bm x(t_0))$. 
When $\varphi_j$ and $\lambda_j$ are perturbed, we have the following convergence
estimate:
\begin{theorem}
  \label{theorem:linear_acc_c}
Consider a linear dynamical system $\frac{\dif \bm x}{\dif t} =\tensor A\bm x$,
where $\bm x\in \mathbb{R}^d$ and $\tensor A\in \mathbb{R}^{d\times d}$. Assume
$\tensor A^\top$ has a full set of eigenvectors at distinct eigenvalues $\lambda_j$.
  Let $\tilde{\varphi}_j(\bm x)$ and $\tilde \lambda_j$ be approximations of 
  $\varphi_j(\bm x)$ and $\lambda_j$, respectively. Consider a scalar observable 
$g(\bm x(t_0+t))=\sum_j^d c_j \vphi_j(\bm x(t_0)) e^{\lambda_j t}$.
  Denote $\varepsilon_{\varphi_0}=\max_{1\leq j\leq d}\vert \varphi_j(\bm x(t_0))
  - \tilde{\varphi}_j(\bm x(t_0)) \vert$
and $\varepsilon_{\lambda}=\max_{1\leq j\leq d} |\lambda_j-\tilde\lambda_j|$. We
have
  \begin{equation}
    \label{eq:bound1}
\left\vert g(\bm x(t_0+t))-\tilde g(\bm x(t_0+t)) \right\vert 
\leq d \Vert \bm c\Vert_{\infty} e^{Re\,\lambda_{\max}t}\left(\max_j |\varphi_j(\bm x(t_0))| \varepsilon_{\lambda} t+
      \varepsilon_{\varphi}e^{\varepsilon_{\lambda}t} \right),
  \end{equation}
  where 
  $\tilde g(\bm x(t_0+t))=\sum_{j=1}^d c_j \tilde\vphi_j(\bm x(t_0)) e^{\tilde\lambda_j t}$
and $Re\,\lambda_{\max}=\max_j Re\,\lambda_j$.
\end{theorem} 
\begin{proof} With the triangle inequality, we have
  \begin{equation}
    \begin{aligned}
      & \left\vert g(\bm x(t_0+t))-\tilde g(\bm x(t_0+t)) \right\vert \\ 
  = &\left \vert \sum_j^d c_j \vphi_j(\bm x(t_0)) e^{\lambda_j t}
  -\sum_j^d c_j \tilde\vphi_j(\bm x(t_0)) e^{\tilde\lambda_j t} \right\vert \\
       \leq &\sum_{j=1}^d |c_j|\cdot \left \vert \vphi_j(\bm x(t_0)) e^{\lambda_j t}-
    \tilde\vphi_j(\bm x(t_0)) e^{\tilde\lambda_j t} \right\vert \\
       \leq & \sum_{j=1}^d |c_j|\cdot \left(\left \vert \vphi_j(\bm x(t_0)) e^{\lambda_j t}\right\vert \left\vert
      1-e^{(\tilde{\lambda}_j-\lambda_j)t}\right\vert 
      +\vert \varphi_j(\bm x(t_0))-\tilde{\varphi}_j(\bm x(t_0)) \vert\cdot \vert e^{\lambda_j t}\vert
      \cdot \vert e^{(\tilde{\lambda}_j-\lambda_j)t}\vert \right) \\
      \leq & d \Vert \bm c\Vert_{\infty} e^{Re\,\lambda_{\max}t} \left(\max_j |\varphi_j(\bm
      x(t_0))|\cdot|1-e^{\varepsilon_{\lambda}t}|+
      \varepsilon_{\varphi_0}e^{\varepsilon_{\lambda}t} \right) .
    \end{aligned}
  \end{equation}
\end{proof}
If the eigen-decomposition solver is accurate, then $\varepsilon_{\varphi_0}$ and
$\varepsilon_{\lambda}$ are very close to zero. Consequently, 
$|1-e^{\varepsilon_{\lambda}t} |\approx \varepsilon_{\lambda}t$, and
$\tilde g$ is very
close to $g$ given accurate Koopman modes $c_j$. For example, when solving 
linear PDEs, $\tensor A$ can be the differentiation
matrix of $\nabla$ or $\Delta$ (or other differential operators in the finite
difference or pseudo-spectral method). Hence, it is promising that when solving
these PDEs based on their semi-discrete form, the time evolution can be
very accurate by ASK. Moreover, since $\varphi_j(\bm x)=\langle \bm x, \bm
w_j\rangle$ in the linear case, we can further bound $\varepsilon_{\varphi}$ and
$|\varphi_j(\bm x(t_0))|$ by $\Vert\bm x(t_0)\Vert_2$ via Cauchy-Schwartz
inequality.

We note that in~\Cref{theorem:linear_acc_c}, Koopman modes $c_j$ are assumed to
be accurate. As shown in~\Cref{subsec:SASK_solution_construction}, the ASK 
method computes these modes by solving linear systems based on the computed 
eigenfunctions $\tilde\varphi_j$ using different initial value $g(\bm x(t_0))$.
The following theorem provides the error estimate in this practical scenario.
\begin{theorem}
  \label{theorem:linear_app_c}
  Given the condition in~\Cref{theorem:linear_acc_c}, if the Koopman modes $c_j$
  are approximated by $\tilde c_j$ that are computed by solving a linear system
  $\tilde{\tensor \Phi} \tilde{\bm c} = \bm g$, where $\tilde\Phi_{ij}=\tilde\varphi_j(\bm x_i)$,
  $g_j = g(\bm x_j)$, and $\bm x_i$ are $d$ different initial values. 
  Let $\varepsilon_{\varphi_{\max}} = \max_{1\leq i\leq d} \max_{1\leq j\leq
  d}|\varphi_j(\bm x_i)-\tilde{\varphi}_j(\bm x_i)|$. Consider 
  $\tilde g(\bm x(t_0+t))=\sum_{j=1}^d \tilde c_j \tilde\vphi_j(\bm x(t_0))
  e^{\tilde\lambda_j t}$, which is an approximation of $g(\bm x(t_0+t))$. Let
  $\delta\tensor\Phi=\tilde{\tensor\Phi}-\tensor\Phi$.
  If 
  $\Vert\tensor\Phi^{-1}\Vert_{\infty}\varepsilon_{\varphi_{\max}}<1/d$,
  we have
    \begin{multline}
\left\vert g(\bm x(t_0+t))-\tilde g(\bm x(t_0+t)) \right\vert \leq \\
d \Vert\bm c\Vert_{\infty} \varepsilon_{\varphi_{\max}} e^{Re\,\lambda_{\max}t}
      \left(\max_j |\varphi_j(\bm x(t_0))|(|1-e^{\varepsilon_{\lambda}t}|+\eta
      e^{\varepsilon_{\lambda}t}) +
      \varepsilon_{\varphi_0}e^{\varepsilon_{\lambda}t}\right),
    \end{multline}
    where $\eta=\frac{\kappa d\varepsilon_{\varphi_{\max}}}{\Vert \tensor \Phi\Vert_{\infty} -
  \kappa d\varepsilon_{\varphi}} $ and   $\kappa=\Vert\tensor\Phi\Vert_{\infty}\cdot\Vert\tensor\Phi^{-1}\Vert_{\infty}$. 
\end{theorem}
\begin{proof}
 The Koopman modes $c_j$ in this case satisfy $\tensor \Phi\bm c=\bm g$, where
  $\Phi_{ij}=\varphi_j(\bm x_i)$. Therefore, 
  $\Vert \delta\tensor\Phi \Vert_{\infty} \leq d\varepsilon_{\varphi_{\max}}$.
  Then, a well-known conclusion in numerical linear algebra indicates that
  \[\dfrac{\Vert \bm c-\tilde{\bm c} \Vert_{\infty}}{\Vert\bm c
  \Vert_{\infty}}\leq\dfrac{\kappa\Vert\delta\tensor\Phi \Vert_{\infty} }{\Vert\tensor\Phi\Vert_{\infty}
  -\kappa\Vert\delta\tensor\Phi \Vert_{\infty} } \leq
  \dfrac{\kappa d\varepsilon_{\varphi_{\max}}}{\Vert \tensor \Phi\Vert_{\infty} -
  \kappa d\varepsilon_{\varphi_{\max}}}, \]
 we have
  \begin{equation}
    \begin{aligned}
      & \left\vert g(\bm x(t_0+t))-\tilde g(\bm x(t_0+t)) \right\vert \\ 
      \leq & \underbrace{\left \vert \sum_j^d c_j \vphi_j(\bm x(t_0)) e^{\lambda_j t}
      -\sum_j^d c_j \tilde\vphi_j(\bm x(t_0)) e^{\tilde\lambda_j t}
      \right\vert}_{I} \\ & +
      \left \vert \sum_j^d  c_j \tilde\vphi_j(\bm x(t_0)) e^{\tilde\lambda_j t}
  -\sum_j^d \tilde c_j \tilde\vphi_j(\bm x(t_0)) e^{\tilde\lambda_j t} \right\vert \\
       \leq & I+ \sum_{j=1}^d |c_j-\tilde c_j|\cdot \left \vert 
    \tilde\vphi_j(\bm x(t_0)) e^{\tilde\lambda_j t} \right\vert \\
      \leq & I + d \eta\Vert\bm c\Vert_{\infty} \left(\max_j
      |\varphi_j(\bm x(t_0))| + \varepsilon_{\varphi_{\max}}\right)
      e^{(Re\,\lambda_{\max}+\varepsilon_{\lambda})t} \\
      \leq & d \Vert \bm c\Vert_{\infty} e^{Re\,\lambda_{\max}t} \left(\max_j |\varphi_j(\bm x(t_0))|\cdot |1-e^{\varepsilon_{\lambda}t}|+
      \varepsilon_{\varphi_0}e^{\varepsilon_{\lambda}t} \right) \\ &+  d \eta\Vert\bm c\Vert_{\infty} \left(\max_j
      |\varphi_j(\bm x(t_0))| + \varepsilon_{\varphi_{\max}}\right)
      e^{(Re\,\lambda_{\max}+\varepsilon_{\lambda})t}  \\
      \leq & d \Vert\bm c\Vert_{\infty} \varepsilon_{\varphi_{\max}} e^{Re\,\lambda_{\max}t}
      \left(\max_j |\varphi_j(\bm x(t_0))|(|1-e^{\varepsilon_{\lambda}t}|+\eta e^{\varepsilon_{\lambda}t}) + \varepsilon_{\varphi_0}e^{\varepsilon_{\lambda}t}\right)
    \end{aligned}
  \end{equation}
\end{proof}
\Cref{theorem:linear_app_c} also holds for general cases, i.e.,
when $\bff$ in \Cref{eq:auto} has a full set of eigenvectors at distinct eigenvalues.

Of note, these preliminary analysis results provide the first glance on the accuracy of the ASK (as well as the SASK) method, which will serve as the foundation of more comprehensive study. 
For the most general cases, we need to consider the error of: 1) approximating global eigenfunctions locally; 2) approximating local eigenfunctions using pre-decided basis functions (polynomials in this work); and 3) the eigen-solver and the linear solver. 
Also, we need to investigate the impact of continuous spectrum in some systems.
These analyses require very systematic study and will be included in our future work.


\section{Numerical Results}
\label{sec:SASK_experiments}


The performance of SASK is demonstrated in this section. To solve a PDE, SASK exploits the semi-discrete form of the PDE, where the spatial discretization is performed by the spectral-collocation method. In this way, the PDE is converted to a high-dimensional ODE system. 

This section exemplifies how SASK is capable of solving PDEs accurately with two linear PDEs and two nonlinear PDEs. Notably, we apply a high-accuracy spectral collocation method to generate the reference solution for the PDEs without a close-form solution. 


Computational efficiency is another focus of our work. To this end, we demonstrate the low computational cost of SASK on the PDEs, using the running time against the error in \Cref{subsec:comparison}. As a comparison, fourth order Runge-Kutta (RK4) is incorporated to solve the semi-discretized form, since it is one of the most frequently used conventional numerical methods.

\subsection{Solving PDEs}
To illustrate SASK's effectiveness, we implemented numerical experiments on four PDEs, the details of which will be exhibited subsequently. Specifically, the spatial discretization is based on the Fourier collocation method
(see e.g.,~\cite{ tang2006spectral, hesthaven2007spectral, trefethen2000spectral}) as we impose \emph{periodic} boundary conditions to all the PDEs. The MATLAB code generating differentiation matrices can be found in~\cite{trefethen2000spectral}. Since the induced ODE systems tend to be high-dimensional, we fix the approximation level $\kappa = 1$ so that the computation cost remains feasible as the number of sparse grid points is $2m+1$. Nevertheless, the results turn out to be accurate with such a low approximation level. Suppose the degree of freedom in space is $m$, then the ODE system is $m$-dimensional. We denote $y \in \mr^m$ the SASK solution at different spatial grid points, and $y^*$ the exact solution (or the reference solution). The performance of SASK is quantified by the relative $L_2$ error defined by $\frac{\Vert y - y^* \Vert_2}{\Vert y^* \Vert_2}$ and $L_\infty$ error defined by $\Vert y - y^* \Vert_\infty$. 

\subsubsection{Advection equation}
We consider the following advection equation with an initial condition:
\begin{equation}
  \begin{aligned}
    &  \dfrac{\partial u}{\partial t} + \dfrac{\partial u}{\partial x} = 0, \quad x\in [0, 1], \\
    &  u(x, 0) = 0.2+\sin(\cos(4 \pi x)).
  \end{aligned}
\end{equation}
The closed-form solution is $u(x, t)=u(x-t, 0)$.
The collocation points in space are set as $x_j=\frac{j}{32}, j=0,1,\dotsc,32$, which leads to a 32-dimensional ODE system. 
SASK applied the following set of parameters: $n = 10, r = 1, \gamma = 0.2$. 
The errors at $T = 10$ are $e_{L_2} = 2.81 \times 10^{-12}$ and $e_{L_\infty} = 3.07 \times 10^{-12}$. 
\Cref{fig:SASK_Advection} illustrates the SASK solutions compared with the exact solutions.
It is also observed that the accuracy did not decrease significantly if $n = 1$. 
This is because the semi-discretized system is linear, and adaptivity is barely triggered. 
\begin{figure}[!ht]
  \centering
  \includegraphics[width=0.5\textwidth]{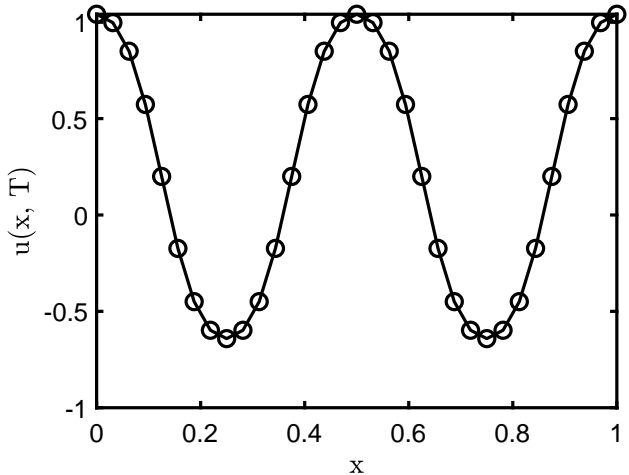}
  \caption{Advection equation: $\medcircle, -$ denote SASK and the reference, respectively.}
  \label{fig:SASK_Advection}
\end{figure}

\subsubsection{Heat equation}
The following is a heat diffusion equation. Here, we also incorporate an initial condition:
\begin{equation}
  \begin{aligned}
    &  \dfrac{\partial u}{\partial t} = \dfrac{1}{80\pi^2}\dfrac{\partial^2 u}{\partial x^2}, \quad x\in [0, 1], \\
    &  u(x, 0) = \sin(4\pi x).
  \end{aligned}
\end{equation}
This equation has a closed-form solution $u(x, t) = \sin(4\pi x) e^{-0.2t}$. 
Specifying parameters $m = 32, n = 10, r = 0.1, \gamma = 0.2$, we obtain SASK numerical solutions at $T = 10$. The comparison between the exact solutions with the numerical solutions is exhibited in~\Cref{fig:SASK_Heat}.  In particular, $e_{L_2} = 8.69 \times 10^{-16}$ and $e_{L_\infty} = 9.08 \times 10^{-15}$.
As in the advection equation case, the semi-discrete form is a linear ODE system, and we observed that SASK achieved high accuracy even when $n=1$, because the adaptivity was barely triggered due to the linearity of the dynamics. 
\begin{figure}[!ht]
  \centering
  \includegraphics[width=0.5\textwidth]{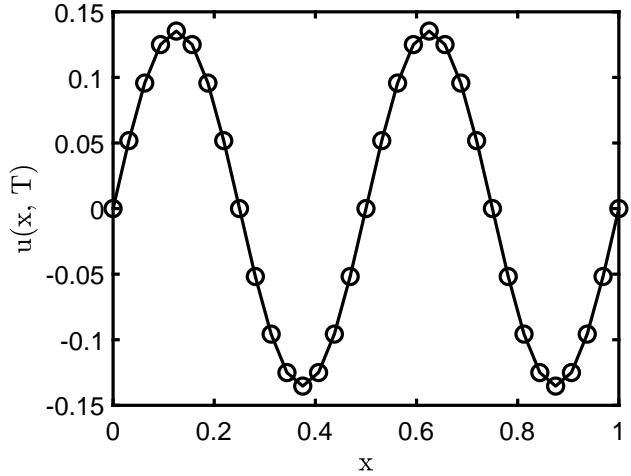}
  \caption{Heat equation: $\medcircle, -$ denote SASK and the reference, respectively.}
  \label{fig:SASK_Heat}
\end{figure}

\subsubsection{Korteweg-de Vries equation}
The next example is the Korteweg-de Vries (KdV) equation with a solitary wave solution:
\begin{equation}
  \label{eq:kdv}
  \begin{aligned}
  \dfrac{\partial u}{\partial t}+\beta u\dfrac{\partial u}{\partial x}+
    \mu\dfrac{\partial^3 u }{\partial x^3}=0, \quad x\in \mathbb{R}, \\
    u(x, 0) = \dfrac{3c}{\beta} \sech^2\left(\dfrac{1}{2}\sqrt{c/\mu} \, x \right).
  \end{aligned}
\end{equation}
The closed-form solution is $u(x,t) = \dfrac{3c}{\beta} \sech^2\left(\dfrac{1}{2}\sqrt{c/\mu}(x-ct) \right)$.
Here, $\beta, \mu, c$ are constants, and $c$ is the wave speed. 
In this test, we set $c = 0.5, \beta = 3, \mu = 9$. 
In general, the solution decays to zeros for $|x|>>1$.
Therefore, numerically we solve this equation in a finite domain $[-p, p]$ with a periodic boundary condition. In this example we set $p=45$. 
Furthermore, as shown in~\cite{fornberg1978numerical, tang2006spectral}, a change of variable step $(x\rightarrow \pi x/p+\pi)$ transforms the solution interval from $[-p, p]$ to $[0, 2\pi]$. 
Consequently, $\beta$ and $\mu$ in~\Cref{eq:kdv} are replaced by $\tilde{\beta} = \frac{\beta \pi}{p}$ and $\tilde{\mu} = \frac{\mu \pi^3}{p^3}$, respectively.
For the spatial discretization, we set $m = 100$ since the behavior of the KdV equation requires finer spatial grids. 
The parameter associated with SASK are $n = 100, r = 0.1, \gamma = 0.8$. 
SASK computed the solutions at $T = 10$, giving the illustration in~\Cref{fig:SASK_KdV}. 
Finally, the errors are $e_{L_2} = 1.60 \times 10^{-4}$ and $e_{L_\infty} = 1.40 \times 10^{-4}$.
\begin{figure}[!ht]
  \centering
  \includegraphics[width=0.5\textwidth]{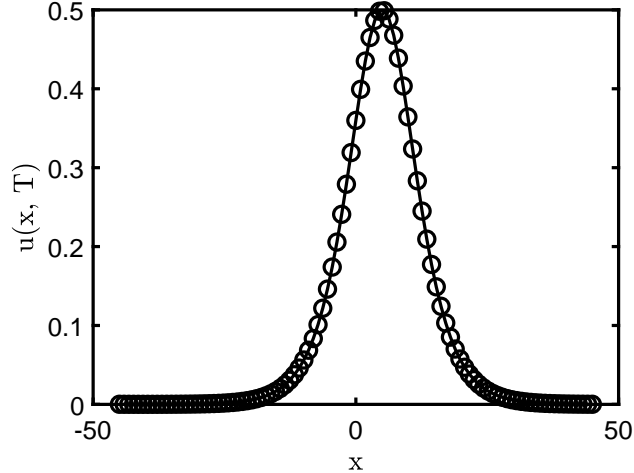}
  \caption{KdV equation: $\medcircle, -$ denote SASK and the reference, respectively.}
  \label{fig:SASK_KdV}
\end{figure}

\subsubsection{Burgers equation}
The last example is a viscous Burgers equation with an initial condition is considered:
\begin{equation}
  \begin{aligned}
    &  \dfrac{\partial u}{\partial t} + u\dfrac{\partial u}{\partial x} = \nu \dfrac{\partial^2 u}{\partial x^2}, \quad x\in [0, 1], \\
    &  u(x, 0) = 0.2+\sin(2\pi x).
  \end{aligned}
\end{equation}
The advection term $uu_x$ is treated in the conservation form, i.e., $\frac{1}{2}(u^2)_x$.
Particularly, the reference solution is obtained by the high-order spectral method. For this example, the degree of freedom in space is $m = 64$, and SASK admits the parameters $n = 50, r = 0.1, \gamma = 1$.
With $\nu=0.005$ and $T=1$, SASK yields $e_{L_2} = 8.85 \times 10^{-4}$ and $e_{L_\infty} = 1.80 \times 10^{-3}$. 
The result is presented in~\Cref{fig:SASK_Burgers}, which indicates that the solution is accurate before the discontinuity is fully developed. Additionally, we will observe the Gibbs phenomenon near the discontinuity, if we keep increasing $T$. 
\begin{figure}[!ht]
  \centering
  \includegraphics[width=0.5\textwidth]{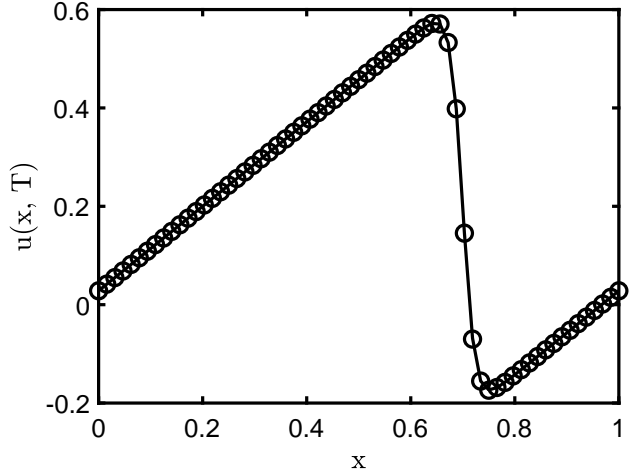}
  \caption{Burgers equation: $\medcircle, -$ denote SASK and the reference, respectively.}
  \label{fig:SASK_Burgers}
\end{figure}

\subsection{Efficiency comparison}
\label{subsec:comparison}

To demonstrate the efficiency of SASK compared with RK4, we target on the running time of the two methods based on the PDEs. Particularly, RK4 leverages the Fourier differentiation matrix to solve the semi-discrete form after spatial discretization. The parameters used in the four PDEs are specified as follows:
\begin{enumerate}[\qquad (a)]
    \item Adevection equation: $m = 32, T = 100, n = 1, r = 1, \gamma= 0.2$;
    \item Heat equation: $m = 32, T = 10, n = 10, r = 0.1, \gamma= 0.5$;
    \item KdV equation: $c = 0.5, \beta = 3, \mu = 9, p = 45, m = 100, T = 10, n = 100, r = 0.1, \gamma= 0.8$;
    \item Burgers equation: $\nu = 0.005, m = 64, T = 1, n = 100, r = 0.1, \gamma= 1$.
\end{enumerate}
We summarize the results in \Cref{tab:PDE_comparison}. The first row of each PDE is the running time, while the second and the third are the relative $L_2$ error and the $L_\infty$ error. 
For the advection equation and the heat equation, the dynamics of their semi-discrete form are linear so SASK needs only a small number of check points and updates to be accurate. 
This results in SASK's higher computational efficiency than RK4's. 
In contrast, Burgers equation and KdV equation require more updates, so it takes SASK more computations to preserve accuracy. 
For the KdV equation, SASK still outperforms RK4, but for the Burgers equation, SASK needs around 10 times the time of RK4 to reach a comparable level of accuracy. 
This is because for the Burgers equation, SASK requires more adaptivity steps to track the evolution of the states, which is probably due to the regularity of the system itself.
Hence, SASK gains an edge over RK4 when $T$ is large, or when the dynamics are not so nonlinear that it allows for a small number of updates. 
\begin{table}[!ht]
  \centering
  \caption{Performance comparison between SASK and RK4 on PDEs: $L_2$ denotes the relative error computed with $2$-norm, while $L_{\infty}$ denotes the error computed with $\infty$-norm. The normalized time is reported.}
    \begin{tabular}{cccc}
    \hline\hline
          &       & SASK  & RK4 \\
    \midrule
    \multirow{3}[1]{*}{Advection Equation} & time   & 1.00 & 21.47 \\
          & $L_2$ & 2.80e-11 & 1.49e-07 \\
          & $L_{\infty}$ & 3.02e-11 & 1.78e-07 \\
          \hline
    \multirow{3}[0]{*}{Heat Equation} & time   & 1.00 & 13.28 \\
          & $L_2$ & 7.57e-15 & 9.24e-15 \\
          & $L_{\infty}$  & 1.19e-15 & 1.85e-15 \\
          \hline
    \multirow{3}[0]{*}{KdV Equation} & time   & 1.00 & 3.97 \\
          & $L_2$  & 1.5972e-04 & 1.5968e-04 \\
          & $L_{\infty}$  & 1.4030e-04 & 1.4090e-04 \\
          \hline
    \multirow{3}[0]{*}{Burgers Equation} & time  & 1.00 & 0.1 \\
          & $L_2$  & 4.94e-04 & 4.88e-04 \\
          & $L_{\infty}$  & 6.13e-04 & 4.96e-04 \\
    \hline\hline
    \end{tabular}%
  \label{tab:PDE_comparison}%
\end{table}%




\section{Conclusion and Discussion}
\label{sec:dicussion_conclusion}

In this work, we propose the sparse-grid-based ASK method to solve autonomous dynamical systems. 
Leveraging the sparse grid method, SASK is an efficient extension of the ASK for
high-dimensional ODE systems. 
Also, we demonstrate SASK's potential of solving PDEs efficiently by solving semi-discrete systems. 
In particular, our numerical results demonstrate that if the semi-discrete form is a linear function of the
discretized solution of the PDE, SASK is very accurate and much more efficient than conventional ODE-solver-based methods because SASK is a high order method for solving dynamical systems. 
Furthermore, by selecting the adaptivity criteria carefully, SASK can solve highly nonlinear PDEs like the KdV equation and deal with large total variation in the solution like the Burgers’ equation. 
Moreover, we only used level-1 sparse grid method in SASK to obtain good results in the illustrative examples. 
In our future work, we will further investigate the selection of adaptivity parameters and the requirement on the accuracy level of the sparse grid method.
Finally, it is possible to use other sampling strategies combined with different basis functions like radial basis functions or activation functions used in neural networks to further enhance the accuracy and efficiency of ASK
for some high-dimensional problems.


\appendix
%

\section{Sparse grid illustration}
\label{append:sparse}
\Cref{fig:sparse_grid} is an illustration of sparse grids and their full grid counterparts.  
\begin{figure}[htbp]
  \centering
  
  \subfloat[$\kappa = 1$]{\includegraphics[width=0.3\textwidth]{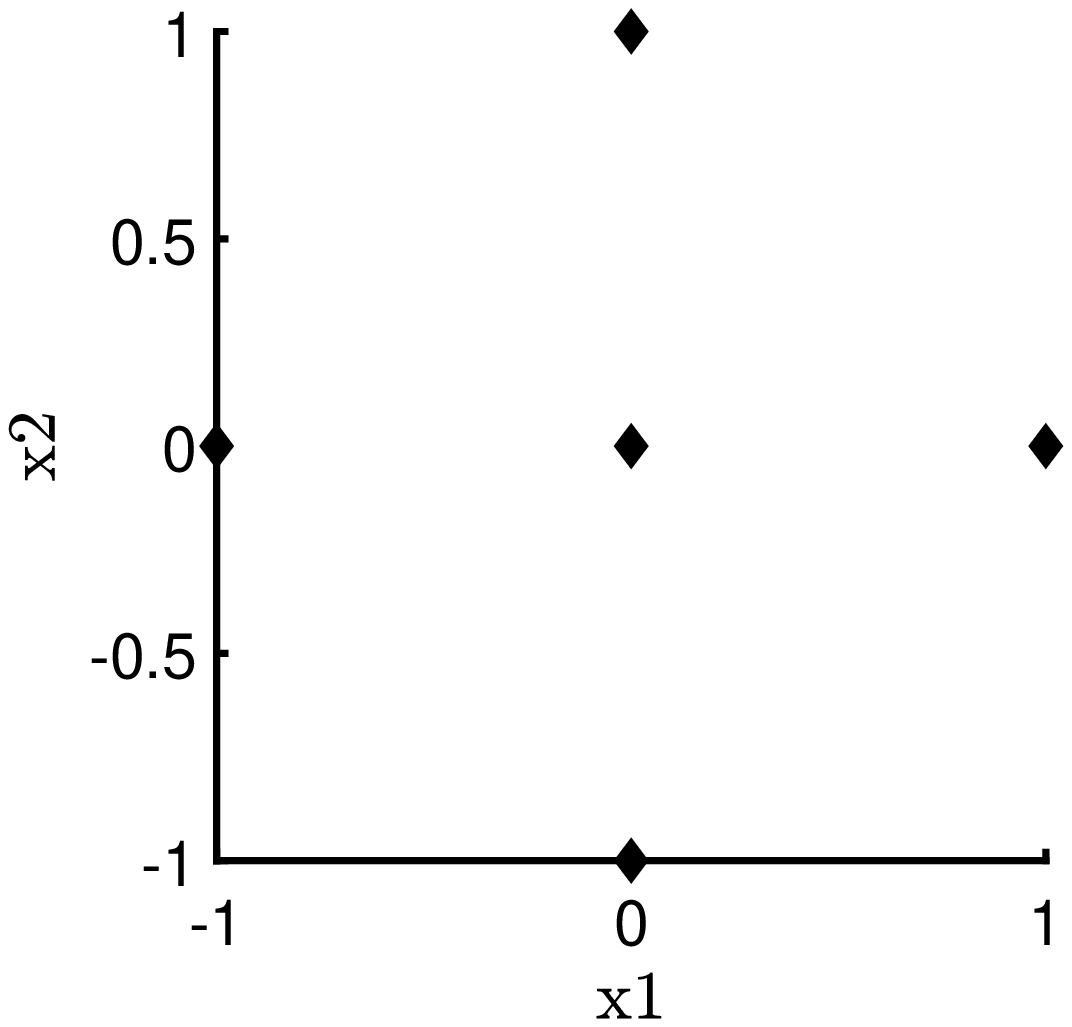}}\quad
  \subfloat[$\kappa = 2$]{\includegraphics[width=0.3\textwidth]{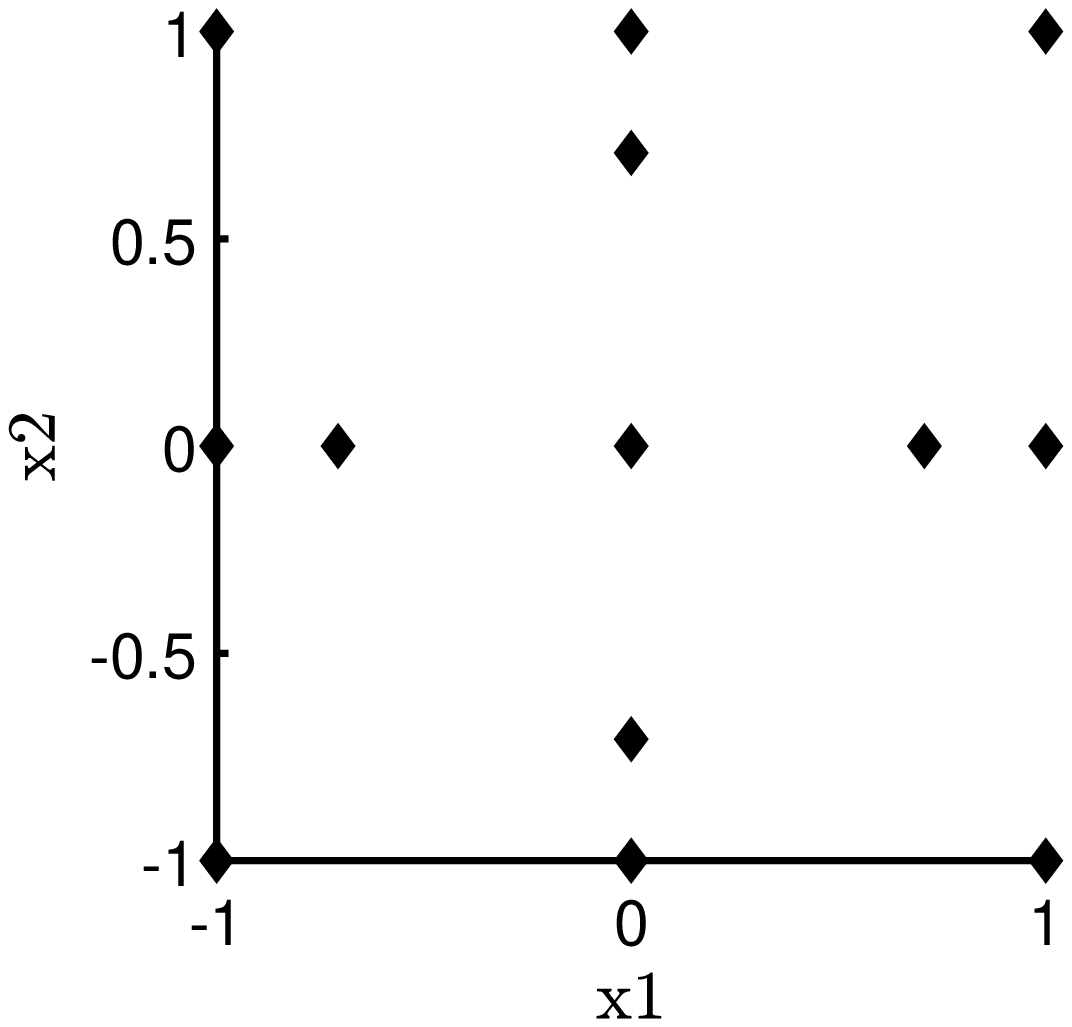}} \quad
  \subfloat[$\kappa = 3$]{\includegraphics[width=0.3\textwidth]{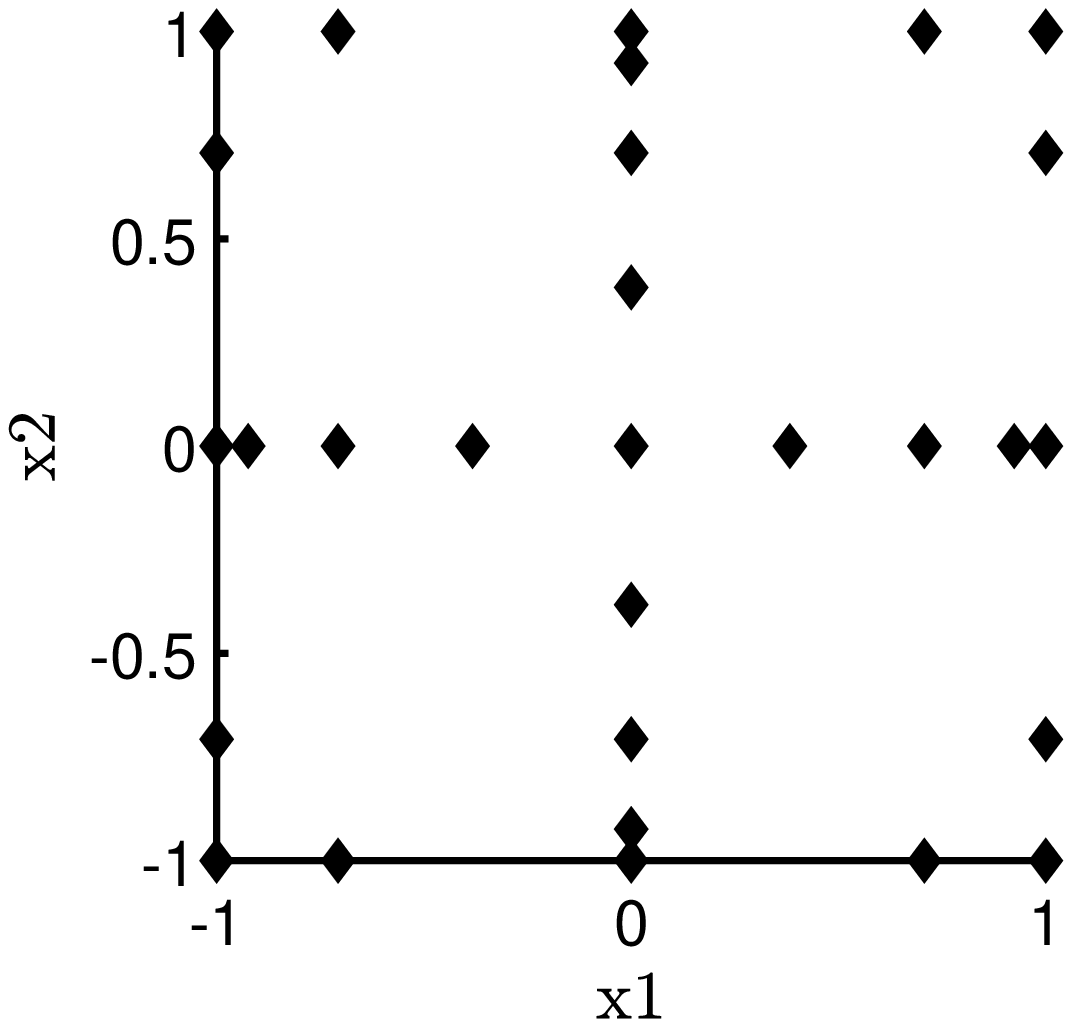}} \\
  \subfloat[$n = 3$]{\includegraphics[width=0.3\textwidth]{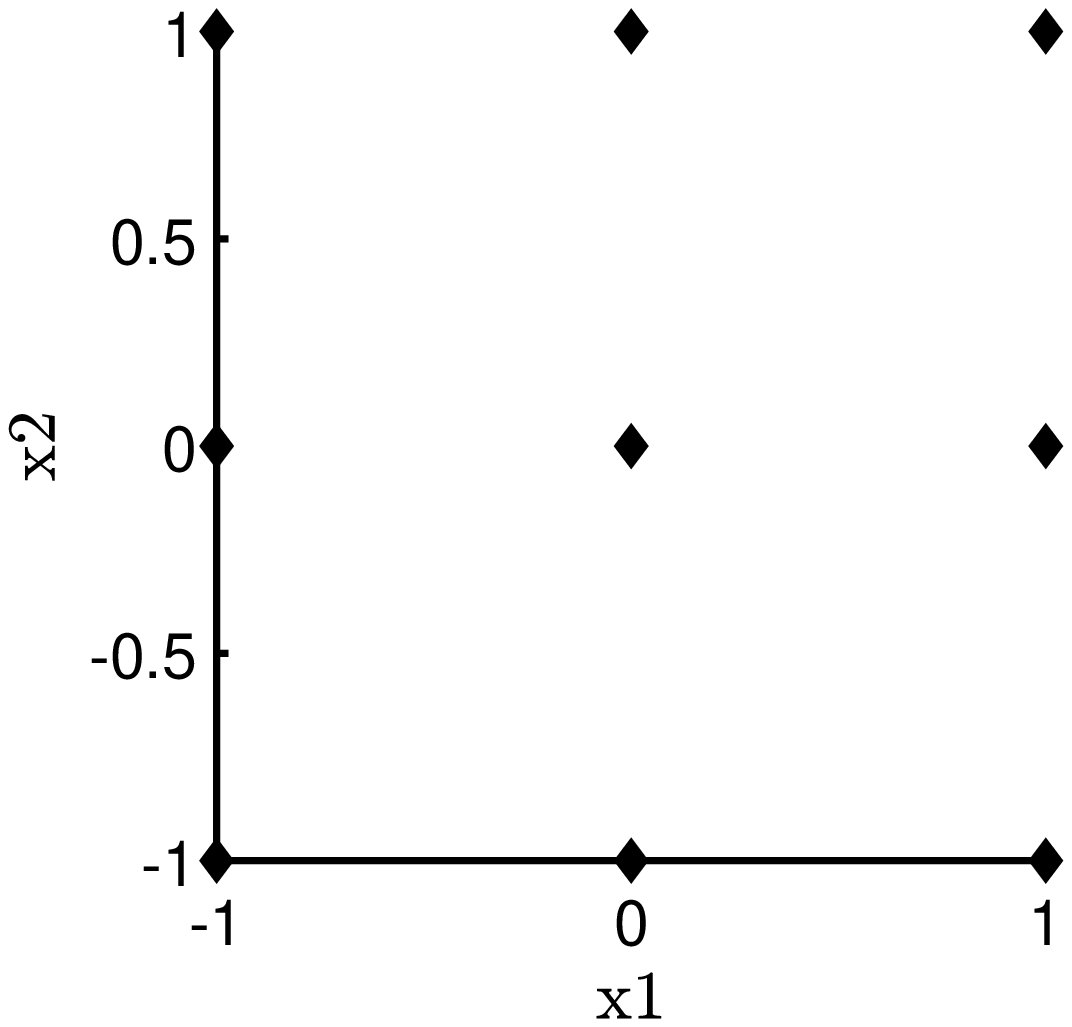}}\quad
  \subfloat[$n = 5$]{\includegraphics[width=0.3\textwidth]{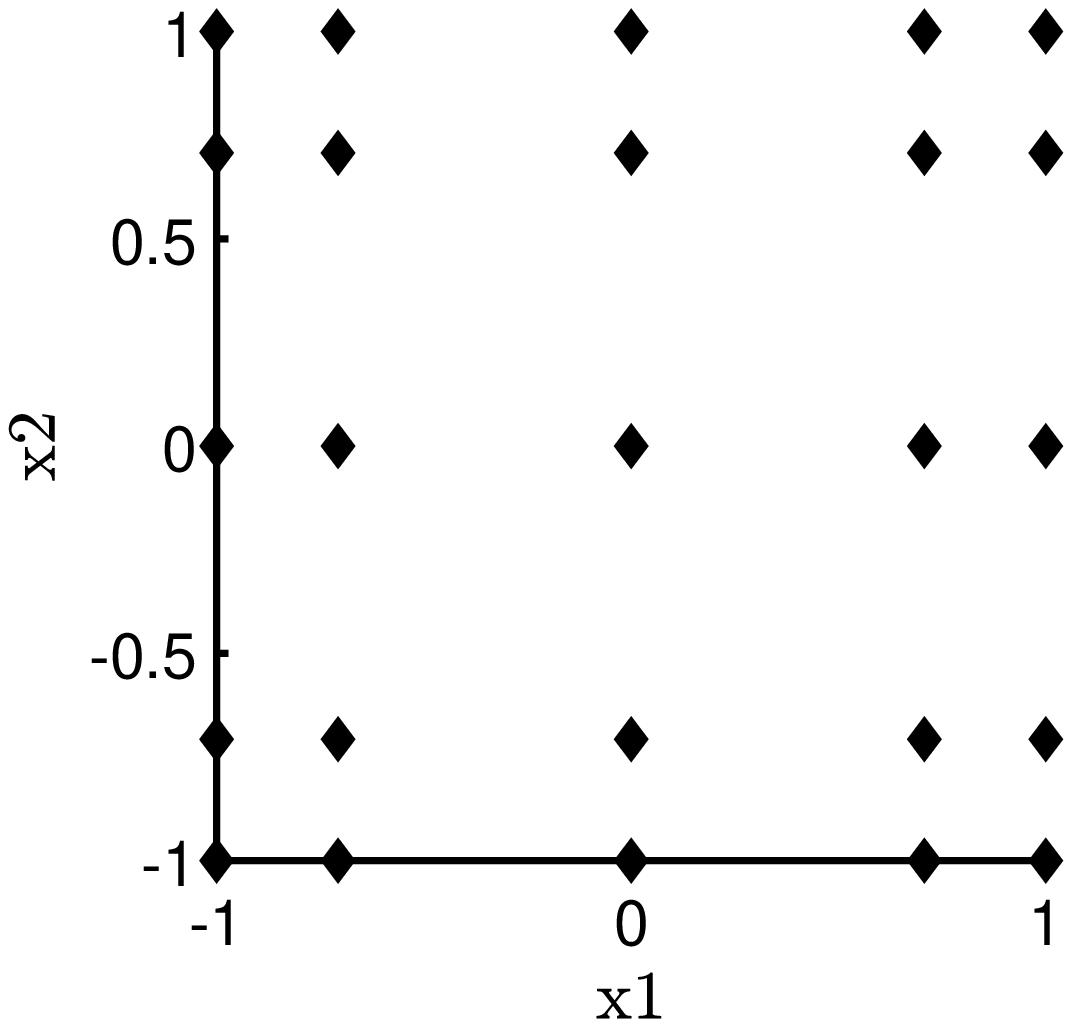}} \quad
  \subfloat[$n = 9$]{\includegraphics[width=0.3\textwidth]{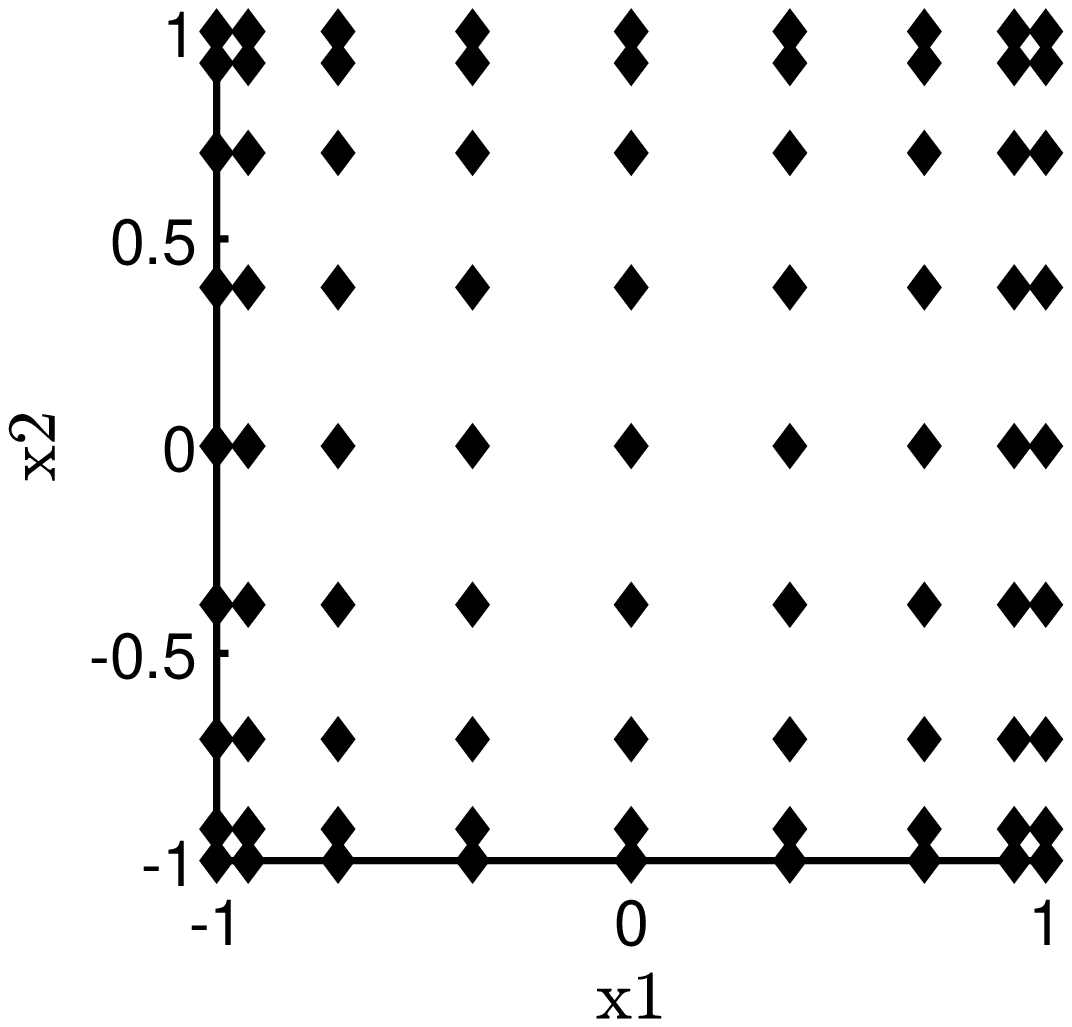}}
  \caption{Sparse grid (first row) - full grid (second row): $n$ denotes the number of points in each dimension.}
  \label{fig:sparse_grid}
\end{figure}
To give a brief view of the growth in the grid size, \Cref{tab:grid_growth} shows that the full grid grows much faster than the sparse grid. Notably, the sparse grid grows surprisingly slow in practice, although its theoretical order is exponential in $d$. 
\begin{table}[htbp]
  \centering
  \caption{Grid Growth: $n$ denotes the number of points in each dimension.}
    \begin{tabular}{ccccccc}
    \hline\hline
    $d$     & $\kappa = 1$ & $n = 3$ & $\kappa = 2$ & $n = 5$ & $\kappa = 3$ & $n = 9$ \\
    \midrule
    2     & 5     & 9     & 13    & 25    & 29    & 81 \\
    3     & 7     & 27    & 25    & 125   & 69    & 729 \\
    4     & 9     & 81    & 41    & 625   & 137   & 6561 \\
    5     & 11    & 243   & 61    & 3125  & 241   & 59049 \\
    \hline\hline
    \end{tabular}%
  \label{tab:grid_growth}%
\end{table}%


\bibliographystyle{plain}
\bibliography{references}
\end{document}